\documentclass[11pt,a4]{article}

\makeatletter
\newcommand{\subjclass}[2][2010]{%
  \let\@oldtitle\@title%
  \gdef\@title{\@oldtitle\footnotetext{#1 \emph{Mathematics subject classification.} #2}}%
}
\newcommand{\keywords}[1]{%
  \let\@@oldtitle\@title%
  \gdef\@title{\@@oldtitle\footnotetext{\emph{Key words and phrases.} #1.}}%
}
\makeatother

\usepackage{subfigure}
\usepackage{graphicx}
\usepackage{sidecap}
\usepackage{amsmath, amsthm, amssymb}
\usepackage{array}
\usepackage{url}
\usepackage{float}

\usepackage{psfrag}
\usepackage{color}
\usepackage{bm}
\graphicspath{{F1Fourier/}{F1Legendre/}{F2Fourier/}{F2Legendre}{sphere-figs/}}

\usepackage[vcentering,dvips]{geometry}
\newtheorem{theorem}{Theorem}[section]
\newtheorem{conjecture}[theorem]{Conjecture}
\newtheorem{lemma}[theorem]{Lemma}
\newtheorem{proposition}[theorem]{Proposition}
\newtheorem{corollary}[theorem]{Corollary}

\newtheorem{remark}[theorem]{Remark}

\newcommand{\E}{{\mathbb{E}} \hspace{.2mm}}

\newcommand{\R}{{\mathbb{R}}}

\DeclareMathOperator{\Tr}{Tr}

 \newcommand{\scalprod}[2]{\left\langle #1,#2 \right\rangle}

 \newcommand{\vertiii}[1]{{\left\vert\kern-0.25ex\left\vert\kern-0.25ex\left\vert #1 
    \right\vert\kern-0.25ex\right\vert\kern-0.25ex\right\vert}}
\newcommand{\Id}{\operatorname{Id}}    

\usepackage{enumitem}

\newcommand{\vct}[1]{\bm{#1}}
\newcommand{\mtx}[1]{\bm{#1}}

\newcommand{\A}{\mtx{A}}
\newcommand{\B}{\mtx{B}}

\newcommand{\U}{\mtx{U}}

\newcommand{\X}{\vct{X}}

\newcommand{\wprec}{\underset{w}{\prec}}

\newcommand{\cH}{\mathcal{H}}

\begin{document}

% paper title
\title{An arithmetic-geometric mean inequality for products of three matrices}

\author{Arie Israel\thanks{A. Israel is with the Mathematics Department at the University of Texas at Austin,   \rm{arie@math.utexas.edu}}, Felix Krahmer\thanks{F. Krahmer is with the Department of Mathematics, Unit M15 Applied Numerical Analysis, Technische Universit{\"a}t M{\"u}nchen, \rm{felix.krahmer@tum.de} },  and Rachel Ward\thanks{R.\ Ward is with the Mathematics Department at the University of Texas at Austin, \rm{rward@math.utexas.edu}}
}


%\date{March 1, 2010; revised July 13, 2010}
%\date{July 4, 2012}

% make the title area

\maketitle

\abstract{Consider the following noncommutative arithmetic-geometric mean inequality: Given positive-semidefinite matrices $\A_1, \dots, \A_n$, the following holds for each integer $m \leq n$:
\begin{equation}
\nonumber
\frac{1}{n^m}\sum_{\mbox{\footnotesize $j_1, j_2, \dots, j_m = 1$}}^{\mbox{\footnotesize $n$}}   \vertiii{ \A_{j_1} \A_{j_2} \dots \A_{j_m} } \geq  \frac{(n-m)!}{n!} \sum_{\mbox{\footnotesize $j_1, j_2, \dots, j_m = 1$} \atop \mbox{\footnotesize all distinct}}^{\mbox{\footnotesize $n$}}   \vertiii{ \A_{j_1} \A_{j_2} \dots \A_{j_m} },
\end{equation}
where $ \vertiii{ \cdot }$ denotes a unitarily invariant norm, including the operator norm and Schatten $p$-norms as special cases.  While this inequality in full generality remains a conjecture, we prove that the inequality holds for products of up to three matrices, $m \leq 3$.  The proofs for $m = 1,2$ are straightforward; to derive the proof for $m=3,$ we appeal to a variant of the classic Araki-Lieb-Thirring inequality for permutations of matrix products.   

}
\section{Introduction}
The arithmetic-geometric mean (AMGM) inequality says that for any sequence of $n$ non-negative real numbers $x_1, x_2, \dots, x_n,$ the  arithmetic mean is greater than or equal to the geometric mean:
\begin{equation}
\nonumber
\frac{x_1 + x_2 + \dots + x_n}{n} \quad \geq \quad \left( x_1 x_2 \dots  x_n \right)^{1/n}.
\end{equation}
This can be viewed as a special case $(m=n)$ of Maclaurin's inequality:
\begin{proposition}
\label{amgm_vector}
If $x_1, \dots, x_n$ are positive scalars and $m \geq n$ then it holds that
$$
\frac{1}{n^m} \sum_{\mbox{\footnotesize $j_1, j_2, \dots, j_m = 1$}}^{\mbox{\footnotesize $n$}} x_{j_1} x_{j_2} \dots x_{j_m} \quad \geq \quad  \frac{1}{{n\choose m}}  \sum_{\mbox{\footnotesize $\Lambda \subset [n];$} \atop \mbox{\footnotesize $| \Lambda | = m$}} x_{j_1} x_{j_2} \dots x_{j_m}.
$$
\end{proposition}
\noindent See \cite{hardy} for more details.  With a slight abuse of notation, we will refer to both of the above inequalities as AMGM inequalities.

\bigskip

Several noncommutative extensions of Proposition \ref{amgm_vector} have been proven for inequalities involving the product of two matrices.  For an overview of these results, we refer the reader to \cite{bhat2}.  %%%%%%
These inequalities are often stated for a general unitarily invariant (UI) norm.  Recall that a norm  $\vertiii{\cdot}$ on $M(d)$,  the space of complex $d \times d$ matrices, is said to be unitarily invariant if for all $\X, \U \in M(d)$ with $\U$ unitary, one has
 $\vertiii{\U \X} = \vertiii{\X \U} =  \vertiii{\X}.$
Examples of UI norms are the Schatten $p$-norms (including the operator norm and the Hilbert-Schmidt norm) and the Ky Fan $k$-norms.  More generally, every UI norm is a symmetric gauge function of the singular values \cite{bhatbook}.  
The first AMGM inequality for products of two matrices appeared in \cite{bhat0}:  If $\A$ and $\B$ are compact operators on a separable Hilbert space, then 
$$
2 \vertiii{ \A^{*}\B } \leq \vertiii{ \A \A^{*} + \B \B^{*} }.
$$
The paper \cite{bhat1} extended this result, showing that for arbitrary $d \times d$ matrices $\A,\B,\X$, and for every unitarily invariant norm, 
$$
2 \vertiii{\A^{*} \X \B } \leq \vertiii{ \A \A^{*} \X + \X \B \B^{*} }.
$$
Most closely related to the results here, Kosaki \cite{kosaki98} showed that for positive-semidefinite matrices $\A$ and $\B$, and for $1/p + 1/q = 1$,
$$
\vertiii{ \A \X \B } \leq \frac{1}{p} \vertiii{ \A^p \X } +  \frac{1}{q} \vertiii{ \X \B^q }.
$$
In the special case $\X = \Id$ and $p = q = 2$, applying this inequality and averaging with respect to the order of $\A$ and $\B$, reproduces our result for the case of two matrices.

\bigskip

\noindent More recently, certain noncommutative extensions of the AMGM inequality have been posed for products of (complex-valued) positive-semidefinite matrices.  The following conjecture was posed by Recht and R{\`e} \cite{recht12}:

\begin{conjecture}
\label{conj:recht}
Let $\A_1, \dots, \A_n$ be positive-semidefinite matrices.  Then the following inequality holds for each $m \leq n$:
\begin{equation}
\nonumber
\left\| \frac{1}{n^m} \sum_{\mbox{\footnotesize $j_1, j_2, \dots, j_m = 1$}}^{\mbox{\footnotesize $n$}} \A_{j_1} \A_{j_2} \dots \A_{j_m} \right\|  \geq  \left\| \frac{(n-m)!}{n!} \sum_{\mbox{\footnotesize $j_1, j_2, \dots, j_m = 1;$} \atop \mbox{\footnotesize \emph{all distinct}}}^{\mbox{\footnotesize $n$}}  \A_{j_1} \A_{j_2} \dots \A_{j_m} \right\|.
\end{equation}
Here, $\| \cdot \|$ denotes the standard operator norm.
\end{conjecture}
\noindent Note that the case $m=1$ is trivially true, as both sides of the inequality are equal.  The conjecture is easily seen to be true for $n=2, m=2$, and in this case something stronger can be said: the symmetrized geometric mean precedes the square of the arithmetic mean in the positive definite order: for any $\A, \B \succeq 0$, 
$$
(\frac{1}{2} \A + \frac{1}{2} \B )^2 - (\frac{1}{2} \A \B + \frac{1}{2} \B \A) \succeq 0.
$$
Recht and R{\`e} also verify that the conjecture holds for general $m$ and $n$ \emph{in expectation} for several classes of random matrices \cite{recht12}.  

\bigskip

\noindent Later, Ducci \cite{ducci12} posed a variant of the noncommutative AMGM conjecture where the matrix operator norm appears inside the summation on either side:
\begin{conjecture}
\label{ducci_conjecture}
Suppose that $\A_1, \dots, \A_n$ are positive-semidefinite matrices.  
The following inequality holds for each $m \leq n$:
\begin{equation}
\label{with_norm}
\frac{1}{n^m} \sum_{\mbox{\footnotesize $j_1, j_2, \dots, j_m = 1$}}^{\mbox{\footnotesize $n$}} \| \A_{j_1} \A_{j_2} \dots \A_{j_m} \|  \geq  \frac{(n-m)!}{n!} \sum_{\mbox{\footnotesize $j_1, j_2, \dots, j_m = 1;$} \atop \mbox{\footnotesize \emph{all distinct}}}^{\mbox{\footnotesize $n$}} \| \A_{j_1} \A_{j_2} \dots \A_{j_m} \|.
\end{equation}
\end{conjecture}

%\noindent One may also generalize the above conjecture to the setting of compact operators acting on a Hilbert space.
%\begin{conjecture}
%\label{ducci_conjecture2}
%Suppose that $\cA_1, \dots, \cA_n$ are compact positive semidefinite operators on a  Hilbert space. The following arithmetic and (symmetrized) geometric mean inequality holds for each $m \leq n$:
%\begin{equation*}
%\frac{1}{n^m} \sum_{\mbox{\footnotesize $j_1, j_2, \dots, j_m = 1$}}^{\mbox{\footnotesize $n$}} \| \cA_{j_1} \cA_{j_2} \dots \cA_{j_m} \|  \geq  \frac{(n-m)!}{n!} \sum_{\mbox{\footnotesize $j_1, j_2, \dots, j_m = 1;$} \atop \mbox{\footnotesize \emph{all distinct}}}^{\mbox{\footnotesize $n$}} \| \cA_{j_1} \cA_{j_2} \dots \cA_{j_m} \|.
%\end{equation*}
%\end{conjecture}
  
\noindent The case $m=1$ of Conjecture \ref{ducci_conjecture} is trivially true, as both sides of the inequality are the same.  For $m=2$, as explained above, the conjecture directly follows, for arbitrary UI norms, from a result of Kosaki \cite{kosaki98} (we provide a proof without assuming this result in Section~\ref{k2}). Given the positive results for the cases $m=1$ and $m=2$, it is tempting to believe that a simple inductive proof could be used to prove the general case.  However, the difficulty in this approach, and with the noncommutative AMGM inequalities in general, is that the product of two positive-semidefinite matrices is not necessarily positive-semidefinite.   The main result of this paper, proved in Section \ref{k3}, is a proof of this conjecture for the case $m=3$.  Again, we prove a more general result which holds for any unitarily-invariant norm:

\begin{theorem}[AMGM inequality for three matrices]
\label{main}
Suppose that $\A_1, \A_2, \dots, \A_n \in M(d)$ are positive-semidefinite.  Let $\vertiii{\cdot}$ be a unitarily invariant norm on $M(d)$.  Then the following AMGM inequality holds:
\begin{equation}
\label{mainineq}
\frac{1}{n^3} \sum_{\mbox{\footnotesize $i,j,k = 1$}}^{\mbox{\footnotesize $n$}} \vertiii{ \A_{i} \A_{j} \A_{k} } \geq  \frac{(n-3)!}{n!} \sum_{\mbox{\footnotesize $i, j,  k =1$}  \atop \mbox{\footnotesize \emph{ all distinct}} }^{\mbox{\footnotesize $n$}} \vertiii{ \A_{i} \A_{j} \A_{k} }.
\end{equation}
\end{theorem}

\begin{remark}
Unitarily invariant norms $\vertiii{\cdot}$ can also be defined in the infinite-dimensional setting of operators on Hilbert spaces.  The subset of compact operators on a Hilbert space ${\cal H}$ having finite norm $\vertiii{ \cdot }$ defines a self-adjoint ideal ${\cal J}$, called the norm ideal associated to $\vertiii{ \cdot }$.  As finite-rank operators are dense in this space, our Theorem \ref{alt-ineq}, and hence Theorem \ref{main}, extends to this more general setting using a limiting argument as presented e.g. in Proposition 2.2 of \cite{CMS05}.

\end{remark}

\noindent To the best of our knowledge, this result represents the first AMGM inequality for products of three matrices.  Our proof uses a variant of the Araki-Lieb-Thirring inequality (Theorem \ref{alt-ineq}).   This inequality is specific to the product of three operators, and it remains an interesting open question whether the result extends to the case $m \geq 4$.

\section{Motivation}  
One can rephrase Conjecture \ref{ducci_conjecture} in terms of comparing the expectations of random matrices formed by sampling with replacement vs. without replacement. Consider a random variable of the form
$X = \| \A_{j_1} \A_{j_2} \dots \A_{j_m} \|,$ 
along with two different probability distributions over the indices $j_{\ell}$: \emph{with replacement} sampling where each index $j_{\ell}$ is drawn uniformly with replacement from the index set $[n]$, and \emph{without replacement} sampling where the indices $j_{\ell}$ are drawn sequentially, uniformly and \emph{without} replacement from $[n]$.  Then $\E_{wr} X,$ the expected value of $X$ corresponding to with-replacement sampling, is equal to the LHS expression of \ref{with_norm}, while $\E_{wor} X,$ the expected value of $X$ corresponding to without-replacement sampling, is equal to the RHS.   From this perspective, the noncommutative AMGM conjectures have interesting implications for stochastic optimization problems.  
For example, Conjecture \ref{ducci_conjecture} would imply that the expected convergence rate of without-replacement sampling is faster than that of with-replacement sampling for randomized iterative solvers such as the \emph{Kaczmarz} method \cite{kaczmarz}.  We repeat the following example from \cite{ducci12, recht12} for completeness.  The Kaczmarz method is a simple and fast method for solving overdetermined consistent least squares problems: solve for $x_{*}$ satisfying $\Phi x_{*} = y,$ where $\Phi \in \mathbb{C}^{n \times d}$ with $n \gg d$.  Let $\varphi_i^{*}$ denote the $i$th row of $\Phi$.  Then, starting from some initial $x_0$, the Kaczmarz algorithm iterates the following recursion until convergence:
$$
x_{k} = x_{k-1} + \frac{y_k - \scalprod{\varphi_{i_k}}{ x_{k-1}}}{ \| \varphi_{i_k} \|_2^2} \varphi_{i_k}.
$$
Set $\A_i = \Id - \varphi_i^{*} \varphi_i / \| \varphi_i \|_2^2$.  Then since $\scalprod{\varphi_i}{x_{*}} = y_i$,
 we may express the residual after $k$ steps of the Kaczmarz algorithm in terms of the matrix product
$$
x_k - x_{*} = \prod_{j=1}^k \A_{i_j} (x_0 - x_{*}).
$$
The residual error can then bounded by
$$
\| \A_{i_1} \A_{i_2} \dots \A_{i_k} (x_0 - x_{*}) \|_2 \leq \| \A_{i_1} \A_{i_2} \dots \A_{i_k} \| \| x_0 - x_{*} \|_2.
$$
Nonasymptotic convergence rates for the Kaczmarz algorithm have been derived \cite{SV09:Randomized-Kaczmarz, needell2013stochastic} in case each row update $i_k$ is selected according to a random update rule, in particular, independently and identically distributed over $[n]$. At the same time, numerical evidence  \cite{recht12} suggest that the convergence rate can be improved by sampling rows independently \emph{without} replacement.  Conjecture \ref{ducci_conjecture} would provide theoretical justification to these observations, showing that with-replacement sampling cannot outperform without-replacement sampling, in expectation with respect to the draw of the row indices.  Since the randomized Kaczmarz algorithm can be viewed as a special case of stochastic gradient descent \cite{needell2013stochastic}, similar remarks about with-replacement vs. without replacement sampling could likely be made for stochastic gradient descent algorithms more generally.

{\section{AMGM inequality for products of two matrices}\label{k2}}
As a warm up to the main result, we first prove the analog of Theorem \ref{main} for products of two positive-semidefinite matrices. Throughout the section, we fix a unitarily invariant norm $\vertiii{\cdot}$ on $M(d)$. In the proof, we will use several basic results about unitarily invariant norms. First, a variant of H\"older's inequality:
\begin{lemma}[Exercise IV.2.7 of \cite{bhatbook}] \label{thm:holder}
Suppose that $p,q,r >0 $ satisfy $\tfrac{1}{p}+\tfrac{1}{q}=\tfrac{1}{r}$. Then for all $\A, \B \in M(d)$ it holds that
 \[
  \vertiii{ |\A\B|^r}^{1/r} \leq \vertiii{ |\A|^{p} }^{1/p}  \vertiii{ |\B|^{q} }^{1/q}.
 \]
\end{lemma}
Here, we have used the standard notation $|\A|:= (\A^*\A)^{1/2}$. It follows from Lemma \ref{thm:holder} that for all  $\A,\B \in M(d)$ one has
\begin{equation} \label{eq:twomat}
  \vertiii{\A\B} =  \vertiii{ |\A \B| } \leq \vertiii{ |\A|^2}^{1/2} \cdot \vertiii{ |\B|^2}^{1/2} \leq \tfrac{1}{2} \vertiii{ |\A|^2 } +\tfrac{1}{2} \vertiii{ |\B|^2},
\end{equation}
where in the first equality we used the polar decomposition and the unitary invariance of $\vertiii{\cdot}$,  and in the last inequality we used the scalar AMGM inequality. 
Similarly, by a repeated application of Lemma~\ref{thm:holder}, for all $\X_1,\X_2,\X_3 \in M(d)$ we have
\begin{align}
  \vertiii{ \X_1 \X_2 \X_3 }  & \leq  \vertiii{ |\X_1|^3}^{1/3} \cdot \vertiii{ |\X_2 \X_3|^{3/2} }^{2/3} \leq \vertiii{ |\X_1|^3}^{1/3} \cdot \vertiii{ |\X_2|^3}^{1/3} \cdot \vertiii{ |\X_3|^3}^{1/3} \nonumber \\ 
  & \leq \tfrac{1}{3} \vertiii{ |\X_1|^3} +\tfrac{1}{3} \vertiii{ |\X_2|^3}+\tfrac{1}{3} \vertiii{ |\X_3|^3}. \label{eq:threemat}
\end{align}
Alternatively, if $\X_2 \succeq 0$ then we obtain
\begin{equation}\label{eq:threematb}
 \vertiii{\X_1 \X_2 \X_3}  \leq \tfrac{1}{2} \vertiii{\X_1 \X_2 \X_1}+\tfrac{1}{2}\vertiii{\X_3 \X_2 \X_3}
\end{equation}
by applying \eqref{eq:twomat} to $\A=\X_1 \X_2^{1/2}$ and $\B=\X_2^{1/2} \X_3$.
% together with the fact that due to the unitary invariance  ({\color{red} Really?}), it holds that $\|\B (\B)^*\|=\|(\B)^* \B \|$.

\bigskip

\begin{proposition}
\label{nis2}
Let $n \geq 2$ and $d \geq 1$, and suppose that $\A_1, \A_2, \dots, \A_n \in M(d)$ are positive-semidefinite.
Then the following arithmetic-geometric mean inequality holds:
\begin{equation*}
\frac{1}{n^2}  \sum_{\mbox{\footnotesize $j,k  =1$}}^{\mbox{\footnotesize $n$}} \vertiii{ \A_{j} \A_{k} }  \geq  \frac{1}{n(n-1)} \sum_{\mbox{ \footnotesize $j, k = 1$}  \atop \mbox{ \footnotesize $j \neq k$}}^{\mbox{\footnotesize $n$}}  \vertiii{ \A_{j} \A_{k} }.
\end{equation*}
\end{proposition}

\begin{proof}
Rearranging and canceling like terms, the inequality reduces to
$$
(n - 1)  \sum_{\mbox{\footnotesize $j  =1$}}^{\mbox{\footnotesize $n$}}   \vertiii{ \A_j^2 } \geq \sum_{\mbox{ \footnotesize $j, k = 1$}  \atop \mbox{ \footnotesize $j \neq k$}}^{\mbox{\footnotesize $n$}}  \vertiii{ \A_{j} \A_{k} },
$$
which is equivalent to 
$$
 \sum_{\mbox{\footnotesize $j  =1$}}^{\mbox{\footnotesize $n$}}   \sum_{\mbox{\footnotesize $k = j+1$}}^{\mbox{\footnotesize $n$}}  \left( \vertiii{ \A_j^2 } + \vertiii{ \A_k^2 } - \vertiii{ \A_{j} \A_{k} } - \vertiii{ \A_k \A_j } \right) \geq 0.
$$
To show that each of the summands is nonnegative, we apply \eqref{eq:twomat} both for $\vertiii{\A_i \A_j }$ and for $\vertiii{\A_j \A_i }$,  obtaining

\begin{align*}
\vertiii{ \A_i^2 } + \vertiii{ \A_j^2 }   \geq  \vertiii{ \A_i \A_j } + \vertiii{ \A_j \A_i }.
\end{align*}
This proves the proposition.

\end{proof}

{\section{AMGM inequality for products of three matrices}\label{k3}}

Before proving the main result, Theorem \ref{main}, we will need to establish an ALT-type inequality for unitarily invariant norms. The classic version of this inequality states:
\begin{theorem}[Araki-Lieb-Thirring \cite{araki1990inequality}]\label{alt1}
For all $r \geq 1$ and $q>0$ one has
\[
\Tr \left[ (\B^r \A^r \B^r)^q \right] \geq \Tr \left[ ( \B \A \B)^{rq} \right].
\]
\end{theorem}
A similar statement holds for the matrix operator norm:
\begin{lemma}\label{alt2}
For all $r \geq 1$ and $s > 0$ one has
\[
\left\| (\B^r \A^r \B^r)^s \right\| \leq \left\| ( \B \A \B)^{rs} \right\|.
\]
\end{lemma}
\begin{proof}
We apply the ALT inequality with $q = s t$, for $t > 0$, to obtain
\[
\Tr \left[ (\B^r \A^r \B^r)^{st} \right]^{1/t} \geq  \Tr \left[ ( \B \A \B)^{rst} \right]^{1/t}.
\]
Take the limit as $t \rightarrow + \infty$. The desired inequality follows as $\displaystyle \lim_{t \rightarrow +\infty} \Tr [ \X^t ]^{1/t} = \| \X \|$ for any  $\X \succeq 0$.
\end{proof}

In fact, we can generalize Lemma \ref{alt2} to any unitarily invariant norm using a standard trick in matrix analysis known as the ``antisymmetric tensor power'' trick.    We defer the proof to the appendix. 
\begin{theorem}\label{alt-ineq}
For any unitarily invariant norm $\vertiii{\cdot}$ on $M(d)$, and for any $r \geq 1$ and $s > 0$, the following holds for 
any two positive-semidefinite matrices $\A$ and $\B$ in $M(d)$:
\[
\vertiii{(\B^r \A^r \B^r)^s} \leq \vertiii{(\B \A \B)^{rs}}.
\]
\end{theorem}
In particular, we will use the following corollary:
\begin{corollary}\label{alt4}
For any two positive-semidefinite matrices $\A,\B \in M(d)$, and for any unitarily invariant norm,
\[
\vertiii{\B^2\A} \geq \vertiii{\B\A\B}.
\]
\end{corollary}
\begin{proof}
Apply Theorem \ref{alt-ineq} with $r=2$ and $s = 1/2$ to obtain
\[
\vertiii{ \sqrt{\B^2 \A^2 \B^2}} \leq \vertiii{ \B \A \B} \qquad \mbox{for any} \; \A,\B \in M(d), \; \A, \B \succeq 0.
\]
Note that $\sqrt{\B^2 \A^2 \B^2} = \sqrt{ (\B^2 \A) (\B^2 \A)^*} = | \B^2 \A |$. By the unitary invariance of the norm $\vertiii{\cdot}$ we thus obtain that
\[
\vertiii{ \B^2 \A} = \vertiii{| \B^2 \A |} \leq \vertiii{ \B \A \B},
\]
as desired. 
\end{proof}

\bigskip

\noindent With Corollary \ref{alt4} in hand, we are now ready to prove the main result, Theorem \ref{main}.

\begin{proof}[Proof of Theorem \ref{main}]
Rearranging and canceling like terms, the desired inequality \eqref{mainineq} is equivalent to 
\begin{equation}
\label{equiv_to}
(n -2)(n-1) \sum_{\mbox{\footnotesize $i, j,  k =1$}  \atop \mbox{\footnotesize not all distinct} }^{\mbox{\footnotesize $n$}} \vertiii{ \A_{i} \A_{j} \A_{k} }  \geq  (3n - 2) \sum_{\mbox{\footnotesize $i, j,  k =1$}  \atop \mbox{\footnotesize all distinct} }^{\mbox{\footnotesize $n$}} \vertiii{  \A_{i} \A_{j} \A_{k} }.
\end{equation}
%This can be re-expressed as
%$$
%\sum_{\Lambda \subset [n] \atop | \Lambda | = 3} \sum_{i, j, k \in \Lambda \atop \text{not all distinct} } \left( 2 \| \A_{i} \A_{j} \A_{k} \| - 
%7 \sum_{i, j,  k \in \Lambda \\  \atop \text{all distinct}} \| \A_{i} \A_{j} \A_{k} \| \right)
%$$
By appeal to Corollary \ref{alt4}, we obtain the bounds
$$
\vertiii{ \A_i^2 \A_j } \geq \vertiii{ \A_i \A_j \A_i }.
$$
By the unitary invariance of $\vertiii{\cdot}$ we know that $\vertiii{\X} = \vertiii{ \X^*}$ for all $\X \in M(d)$. Thus, by the same reasoning as above, we have $\vertiii{ \A_i \A_j^2} = \vertiii{\A_j^2 \A_i} \geq \vertiii{ \A_j \A_i \A_j }$.  Applying these lower bounds to the LHS of \eqref{equiv_to} gives:
\begin{align}
& \text{LHS of } \eqref{equiv_to} \nonumber \\
&= (n-2)(n-1) \left[  \sum_{\mbox{\footnotesize $j  =1$}}^{\mbox{\footnotesize $n$}} \vertiii{ \A_j^3}  + \sum_{\mbox{\footnotesize $i, j =1$}  \atop \mbox{\footnotesize $i \neq j$} }^{\mbox{\footnotesize $n$}}  \left( \vertiii{ \A_{i}^2 \A_{j}} +  \vertiii{ \A_{i} \A_{j}^2} \right) + \sum_{\mbox{\footnotesize $i, j =1$}  \atop \mbox{\footnotesize $i \neq j$} }^{\mbox{\footnotesize $n$}} \vertiii{ \A_{i} \A_{j} \A_{i}} \right] \nonumber \\
&\geq (n-2)(n-1) \sum_{\mbox{\footnotesize $j  =1$}}^{\mbox{\footnotesize $n$}} \vertiii{ \A_j^3 }+ 3(n-2)(n-1)\sum_{\mbox{\footnotesize $i, j =1$}  \atop \mbox{\footnotesize $i \neq j$} }^{\mbox{\footnotesize $n$}}   \vertiii{ \A_{i} \A_{j} \A_{i} }. \label{maineq}
\end{align}
We apply \eqref{eq:threematb} to bound the expressions $\vertiii{ \A_{i} \A_{j} \A_{k}}$ and $\vertiii{ \A_{k} \A_{j} \A_{i} }$, thus obtaining
\[
\vertiii{ \A_{i} \A_{j} \A_{i} } + \vertiii{ \A_{k} \A_{j} \A_{k} } \geq \vertiii{ \A_{i} \A_{j} \A_{k}} + \vertiii{ \A_{k} \A_{j} \A_{i} }.
\]
Summing this estimate over pairwise distinct $i$,$j$,$k$ gives
$$
2 (n-2) \sum_{\mbox{\footnotesize $i, j =1$}  \atop \mbox{\footnotesize $i \neq j$} }^{\mbox{\footnotesize $n$}}  \vertiii{ \A_{i} \A_{j} \A_{i} } \geq 2\sum_{\mbox{\footnotesize $i, j,  k =1$}  \atop \mbox{\footnotesize all distinct} }^{\mbox{\footnotesize $n$}} \vertiii{ \A_{i} \A_{j} \A_{k}}.
$$
Continuing the bound from \eqref{maineq}, we have
\begin{align}
\label{new_bound}
\text{LHS of } \eqref{equiv_to} \geq (n-2)(n-1) \sum_{\mbox{\footnotesize $j  =1$}}^{\mbox{\footnotesize $n$}}  \vertiii{ \A_j^3 } + 3(n-1)  \sum_{\mbox{\footnotesize $i, j,  k =1$}  \atop \mbox{\footnotesize all distinct} }^{\mbox{\footnotesize $n$}}  \vertiii{ \A_{i} \A_{j} \A_{k} }. \quad \quad \quad  \quad \quad \quad \quad \quad \quad
\nonumber
\end{align}
The desired inequality \eqref{equiv_to} will follow if the RHS above is greater or equal to the RHS of \eqref{equiv_to}.  Rearranging and canceling like terms, this is equivalent to asking whether
$$
(n-2)(n-1) \sum_{\mbox{\footnotesize $j  =1$}}^{\mbox{\footnotesize $n$}}  \vertiii{ \A_j^3 } \quad  \geq \quad \sum_{\mbox{\footnotesize $i, j,  k =1$}  \atop \mbox{\footnotesize all distinct} }^{\mbox{\footnotesize $n$}} \vertiii{ \A_{i} \A_{j} \A_{k}}.
$$
But indeed this inequality follows directly from averaging \eqref{eq:threemat} with $(\X_1,\X_2,\X_3) = (\A_{i}, \A_{j}, \A_{k})$ over all choices of pairwise distinct $i,j,k$.
% \begin{align}
% (n-2)(n-1) \sum_{\mbox{\footnotesize $j  =1$}}^{\mbox{\footnotesize $n$}}  \| \cA_j^3 \| &= (n-2)(n-1) \sum_{\mbox{\footnotesize $j  =1$}}^{\mbox{\footnotesize $n$}}  \| \cA_j \|^3  \nonumber \\
% &\geq  \sum_{\mbox{\footnotesize $i, j,  k =1$}  \atop \mbox{\footnotesize all distinct} }^{\mbox{\footnotesize $n$}} \| \cA_{i} \| \| \cA_{j} \| \| \cA_{k} \| \nonumber \\
% &\geq   \sum_{\mbox{\footnotesize $i, j,  k =1$}  \atop \mbox{\footnotesize all distinct} }^{\mbox{\footnotesize $n$}} \| \cA_{i} \cA_{j} \cA_{k} \|, \nonumber
% \end{align}
% where the first inequality uses Muirhead's theorem (Theorem \ref{muirhead}) with the scalar quantities $x_1 = \| \cA_1 \|, x_2 = \| \cA_2 \|, \dots, x_n = \| \cA_n \|,$ and sequences $(a_1,a_2, \dots, a_n) = (3,0,\dots, 0)$ and
% $(b_1, b_2, b_3, b_4, \dots, b_n) = (1,1,1,0,\dots, 0)$. 
This finishes the proof.
\end{proof}

\section*{Discussion}
There are certain difficulties that arise in the case $m=4$ which prevent us from extending the proof technique of Theorem \ref{main}, even in the setting of matrices.  ``Loopy" terms such as $\vertiii{  \A_1 \A_2 \A_3 \A_2 }$ start to appear at $m=4$, and it is not clear how to pair these  terms together to obtain a lower bound of distinct-term products, $\vertiii{  \A_1\A_2 \A_3 \A_4 }$.  Simple inductive arguments building on the cases $m=1,2,3$ are also difficult as the product of two positive-semidefinite operators  is not necessarily positive-semidefinite, and there are not many inequalities concerning products of four positive-semidefinite operators.  This remains a very compelling open problem.

\section*{Acknowledgments}

We are thankful to Afonso Bandeira, Dustin Mixon, Deanna Needell, Ben Recht, and Robert Schaback for helpful conversations on this topic.
Krahmer was supported by the German Science Foundation (DFG) in the context of the Emmy Noether Junior Research Group KR 4512/1-1 (RaSenQuaSI).  Ward was partially supported by an NSF CAREER award, DOD-Navy grant N00014-12-1-0743, and an AFOSR Young Investigator Program award.

%\begin{thebibliography}{1}
\bibliography{AMGM}
\bibliographystyle{abbrv}

\section{Appendix: Proof of Theorem \ref{alt-ineq}}

Given $\A,\B \in M(d)$ with $\A,\B \succeq 0$, and given $r \geq 1$ and $s > 0$, our aim in this section is to verify the ALT-type inequality $\vertiii{(\B^r \A^r \B^r)^s} \leq \vertiii{(\B \A \B)^{rs}}$ for a general unitarily invariant norm $\vertiii{\cdot}$ on $M(d)$. The proof uses techniques from matrix analysis and we will make frequent use of results in \cite{bhatbook}. For the rest of this section, all matrices are assumed to be positive-semidefinite. 

\subsection{Antisymmetric tensorization}

In what follows, we write $\lambda_1(\X) \geq \lambda_2(\X) \geq  \lambda_3(\X) \geq \cdots ,$ to denote the sequence of eigenvalues of a matrix $\X$, written in non-increasing arrangement.

For $1 \leq k \leq d$, we denote the space $\cH_k = \wedge^k \R^d$ for the order-$k$ antisymmetric tensor power of $\R^d$, which consists of all formal real linear combinations of symbols of the form
\[
x_1 \wedge x_2 \wedge \cdots \wedge x_k,
\]
where $x_1,\cdots,x_k \in \R^d$. We regard  $\cH_k$  as a vector space in such a way that addition and scalar multiplication are multilinear in the variables $(x_1,\cdots,x_k)$. Namely,
\[
\left\{
\begin{aligned}
&x_1 \wedge \cdots \wedge x_l \wedge  \cdots \wedge x_k \;\; + \;\;  x_1 \wedge  \cdots \wedge x_l' \wedge \cdots \wedge x_k  \;\; = \;\; x_1 \wedge \cdots \wedge (x_l + x_l') \wedge \cdots \wedge x_k. \\
&c \cdot ( x_1 \wedge \cdots \wedge x_l \wedge \cdots \wedge x_k ) = x_1 \wedge \cdots \wedge (c x_l) \wedge \cdots \wedge x_k.
\end{aligned}
\right.
\]
We also require that the wedge product $\wedge$ is antisymmetric with respect to coordinate interchanges. Namely, for all $1 \leq i < j \leq k$ we have
\[
x_1 \wedge \cdots \wedge x_i \wedge \cdots \wedge x_j \wedge \cdots \wedge x_k = - x_1 \wedge \cdots \wedge x_j \wedge \cdots \wedge x_i \wedge \cdots \wedge x_k.
\]
From this definition one sees that the dimension of $\cH_k$ is ${d \choose k}$.

Given a sequence of matrices $\A_1,\cdots,\A_k \in M(d)$ -- which we regard as linear operators on $\R^d$ -- we form the antisymmetric tensor matrix $\A_1 \wedge \cdots \wedge \A_k$, which represents a linear operator on the tensor space $\cH_k$ defined by
\[
(\A_1  \wedge \cdots \wedge \A_k)(x_1  \wedge \cdots \wedge x_k) = (\A_1 x_1)  \wedge \cdots \wedge (\A_k x_k) \quad \mbox{for} \; x_1 \wedge \cdots \wedge x_k \in \cH_k.
\]
Given a matrix $\A \in M(d)$, we denote 
\[
\wedge^k \A := \overbrace{\A \wedge \A \wedge \cdots \wedge \A}^{k \; \mbox{times}}, \quad \mbox{a linear operator on } \cH_k.
\] 
We present a few basic properties of antisymmetric tensorization. Proofs of properties 1-5, along with a further discussion, are found in Section I.5 of \cite{bhatbook}.
\begin{enumerate}
\item $(\wedge^k \A ) ( \wedge^k \B) = \wedge^k (\A \B)$.
\item $(\wedge^k \A)^* = \wedge^k \A^*$.
\item $(\wedge^k \A)^{-1} = \wedge^k \A^{-1}$ if $\A$ is invertible.
\item If $\A$ is unitary or positive-semidefinite then so is $\wedge^k \A$.
\item If $\lambda_{i_1}, \lambda_{i_2}, \cdots, \lambda_{i_k}$ are eigenvalues of $\A$ associated to linearly independent eigenvectors $u_{i_1},u_{i_2},\cdots, u_{i_k}$, and if  $i_1 > i_2 > \cdots > i_k$, then $\lambda_{i_1} \cdot \lambda_{i_2} \cdots \lambda_{i_k}$ is an eigenvalue of $\wedge^k \A$ associated to the eigenvector $u_{i_1} \wedge u_{i_2} \wedge \cdots \wedge u_{i_k}$. In particular, if $\A$ is diagonalizable with eigenvalues $\lambda_j$, $1 \leq j \leq d$, then $\wedge^k \A$ is diagonalizable with eigenvalues $\lambda_{i_1} \cdots \lambda_{i_k}$, $1 \leq i_1 < \cdots < i_k \leq d$.
\item If $\A$ is positive-semidefinite then $(\wedge^k \A)^s = \wedge^k \A^s$ for all $s > 0$.
\end{enumerate}
\emph{Proof of property 6:}  Fix a complete set of eigenvalues $\lambda_1 \geq \cdots \geq \lambda_d \geq 0$ and eigenvectors $u_1,\cdots,u_d$ for $\A$. The eigenvalues of $\A^s$ are given by $\lambda_1^s,\cdots,\lambda_d^s$ and the eigenvectors are unchanged. According to property 5 we know that the eigenvalues/eigenvectors for $\wedge^k \A^s$ are given by $\lambda_{i_1}^s \cdots \lambda_{i_k}^s$ and $u_{i_1} \wedge \cdots \wedge u_{i_k}$, where $1 \leq i_1 < \cdots < i_k \leq d$. On the other hand, by property 5 we know the eigenvalues of $\wedge^k \A$ are $\lambda_{i_1} \cdots \lambda_{i_k}$ associated to the eigenvectors $u_{i_1} \wedge \cdots \wedge u_{i_k}$; hence the eigenvalues of $(\wedge^k \A)^s$ are $\lambda_{i_1}^s \cdots \lambda_{i_k}^s$ associated to the same eigenvectors. Therefore, it holds that  $(\wedge^k \A)^s = \wedge^k \A^s$.

\begin{lemma}\label{alt3}
For all $r \geq 1$ and $s > 0$ one has
\[
\prod_{i=1}^k \lambda_i \left( (\B^r \A^r \B^r)^s \right) \leq \prod_{i=1}^k \lambda_i \left( \B \A \B)^{rs} \right)
\qquad \mbox{for all} \; k=1,2,\cdots,d.\]
\end{lemma}

\begin{proof}
We apply the norm inequality in Lemma \ref{alt2} to the tensor matrices $\wedge^k \A$ and $\wedge^k \B$. Thus, $\| \left( (\wedge^k \B)^r (\wedge^k \A)^r (\wedge^k \B)^r \right)^s \| \leq \| \left( (\wedge^k \B) \cdot (\wedge^k \A) \cdot (\wedge^k \B) \right)^{rs} \|$. Applying subsequently properties 1 and 6, we obtain $\| \wedge^k (\B^r \A^r \B^r)^s \| \leq \| \wedge^k (\B \A \B)^{rs} \|$. Recall that the eigenvalues of $(\B^r \A^r \B^r)^s$ are ordered in non-increasing arrangement: $\lambda_1((\B^r \A^r \B^r)^s) \geq \lambda_2((\B^r\A^r\B^r)^s) \geq \cdots \geq \lambda_d((\B^r\A^r\B^r)^s)$. Therefore, due to property 5, the operator norm of $\wedge^k  (\B^r \A^r \B^r)^s$, which is the largest eigenvalue of $\wedge^k  (\B^r \A^r \B^r)^s$ is
\[
\lambda_1( (\B^r \A^r \B^r)^s) \cdots \lambda_k ( ( \B^r \A^r \B^r)^s).
\]
By similar reasoning, the operator norm of $\wedge^k  (\B \A \B)^{rs}$ is 
\[
\lambda_1( (\B \A \B)^{rs}) \cdots \lambda_k ( ( \B \A \B)^{rs}).
\]
This establishes the desired inequality.
\end{proof}

\subsection{Weak Majorization}
Recall that $\alpha = (\alpha_1,\alpha_2,\cdots,\alpha_d) \in \R^d$ is \emph{weakly majorized} by $\beta = (\beta_1,\beta_2,\cdots,\beta_d) \in \R^d$ -- written in shorthand as $\alpha \wprec \beta$ -- if for all $k =1,2,\cdots,d$ it holds that
\[
\sum_{i=1}^k \alpha_i \leq \sum_{i=1}^k \beta_i.
\]
We borrow the following lemma from \cite{bhatbook}; see Exercise II.3.5.
\begin{lemma}\label{weak_major}
Assume that $x = (x_i)_{i=1}^d$ and $y = (y_i)_{i=1}^d$ satisfy $x_1 \geq x_2 \geq \cdots \geq x_d \geq 0$ and $y_1 \geq y_2 \geq \cdots \geq y_d \geq 0$. If additionally
\[
\prod_{i=1}^k x_i \leq \prod_{i=1}^k y_i \qquad \mbox{for all} \; k=1,\cdots,d,
\]
then $x \wprec y$.
\end{lemma}

From Lemma \ref{alt3} and Lemma \ref{weak_major} we obtain a weak-majorization inequality between the eigenvalues of the matrices of interest. Namely, for any $r \geq 1$ and $s > 0$ we have
\begin{equation}
\label{weak_major1}
(\lambda_j((\B^r\A^r\B^r)^s))_{j=1}^d \wprec (\lambda_j(\B\A\B)^{rs})_{j=1}^d.
\end{equation}

By the Fan Dominance Theorem (see Theorem IV.2.2 in \cite{bhatbook}), the weak majorization in \eqref{weak_major1} implies the desired estimate, Theorem \ref{alt-ineq}.

\end{document}